\documentclass[12pt]{amsart}
\usepackage{amssymb, amsmath, amsfonts}
\usepackage[raiselinks=false,colorlinks=true,citecolor=blue,urlcolor=blue,linkcolor=blue,bookmarksopen=true,pdftex]{hyperref}
\newtheorem{theorem}{Theorem}[section]

\newtheorem{lemma}[theorem]{Lemma}

\newcommand{\bz}{\mathbb{Z}}
\newcommand{\bq}{\mathbb{Q}}

\newcommand{\bc}{\mathbb{C}}

\newcommand{\wt}{\widetilde}

\newcommand{\SL}{\textrm{SL}}

\newcommand{\Flag}{\textrm{Flag}}

\begin{document}
\baselineskip=15.5pt
\title[Maps between complex Grassmann manifolds]{Maps between certain complex Grassmann manifolds}  
\author[P. Chakraborty]{Prateep Chakraborty}
\author[P. Sankaran]{Parameswaran Sankaran}
\address{The Institute of Mathematical Sciences, CIT
Campus, Taramani, Chennai 600113, India}
\email{prateepc@imsc.res.in}
\email{sankaran@imsc.res.in}

\date{}

\begin{abstract}   Let $k,l,m,n$ be positive integers such that $m-l\ge l>k, m-l>n-k\ge k$ and $m-l>2k^2-k-1$. 
Let $G_{k}(\mathbb{C}^n)$ denote the Grassmann manifold of $k$-dimensional vector subspaces of $\bc^n$.  
We show that any continuous map $f:G_{l}(\bc^m)\to G_{k}(\mathbb{C}^n)$ is rationally null-homotopic.   
As an application, we show the existence of a point $A\in G_{l}(\bc^m)$ such that the vector space 
$f(A)$ is contained in $A$; here $\mathbb{C}^n$ is regarded as a vector subspace of $\mathbb{C}^m\cong \bc^n\oplus\bc^{m-n}.$  
\end{abstract}
\maketitle
%%%%%%%%%%%%%%%%%%%%%%%%%%%%
%\newpage
\section{Introduction}

Let $U(n)\subset GL(n,\mathbb{C})$ denote the unitary group and let $G_{k}(\bc^n)$ denote the homogeneous space $G_{n,k}=U(n)/U(k)\times U(n-k).$  The smooth manifold $G_{k}(\bc^n)$ is the complex Grassmann manifold of $k$-dimensional vector subspaces of $\bc^n$.  It is simply connected and has the structure of a smooth projective variety of (complex) dimension is $k(n-k)$.  To simplify notation we shall hereafter write $G_{n,k}$ to mean $G_k(\bc^n)$  
since we will only be concerned with complex Grassmann manifolds in this paper.
 
The purpose of this note is to prove the following theorem.

\begin{theorem}\label{cohomology}
Let $1\le k\le \lfloor n/2\rfloor$, $1\le l\le \lfloor m/2\rfloor$ and $k<l$, where $m,n$ are positive integers 
such that $m-l>n-k$.  Suppose that $m-l\geq2k^2-k-1$ or $1\le k\le 3$.  Then any homomorphism of 
graded rings $\phi:H^*(G_{n,k};\mathbb{Z})\to H^*(G_{m,l};\mathbb{Z})$ vanishes in positive dimensions. 
\end{theorem}

As a corollary to the above theorem we obtain the following result on the homotopy classification 
of maps between the complex Grassmann manifolds.    

\begin{theorem} \label{homotopy} Let $l,k,m,n$ be as in the above theorem.  Then 
the set $[G_{m,l},G_{n,k}]$ of homotopy classes of maps is finite and moreover each homotopy 
class is rationally null-homotopic.  
\end{theorem}

As another application of  Theorem \ref{cohomology} we obtain the following {\it invariant subspace 
theorem.}  See \cite{mukh-sank} for an analogous result for real Grassmann manifolds.  
We shall regard $\mathbb{C}^n$ as a subspace of $\mathbb{C}^m$ consisting of vectors with 
last $m-n$ coordinates zero. Thus, if $y\in G_{n,k}$ and $x\in G_{m,l}$ it is meaningful to 
write $y\subset x$.

\begin{theorem} \label{invariantsubspace}
Let $f:G_{m,l}\to G_{n,k}$ be any continuous map where $l,k,m,n$ are as in Theorem \ref{cohomology}.  
Then there exists an element $x\in G_{m,l}$ such that $f(x)\subset x$. 
\end{theorem}

We point out that 
the classification of self-maps of a complex Grassmann manifold has been studied in terms of their induced 
endomorphisms of the cohomology algebra by several authors. See \cite{oniel}, \cite{brewster-homer}, \cite{glover-homer}, \cite{hoffman}, \cite{hoffman-homer}. Similar study of maps between two {\it distinct} (real) 
Grassmann manifolds seems to have been initiated in  \cite{ks}.  
Sankaran and Sarkar  \cite{ss} have 
studied the existence (or non-existence) of maps of non-zero degree  between two different complex (resp. quaternionic) 
Grassmann manifolds of the same dimension.  The same problem for oriented real Grassmann manifolds has been settled by Ramani and Sankaran \cite{rs}. 

Our methods are straightforward.  To prove Theorem \ref{cohomology}, we reduce the problem to one 
about endomorphism of 
the cohomology of a certain Grassmann manifold and appeal to a well-known result 
of Glover and Homer \cite{glover-homer}.   Theorem \ref{homotopy} is proved using a result 
due to Glover and Homer \cite{glover-homer2}, namely, 
any map between any two complex Grassmann manifolds---indeed complex flag manifolds---is formal. Our approach to the proof of Theorem \ref{invariantsubspace} 
is similar in spirit to that of \cite[Theorem 1.1]{mukh-sank}.

It has been conjectured that if $\phi$ is any endomorphism of the 
graded $\mathbb{Q}$-algebra $H^*(G_{n,k};\mathbb{Q})$ which is vanishes on $H^2(G_{n,k};\mathbb{Z})$, then $\phi$ vanishes 
in all positive degree.  See \cite{glover-homer}.  Our proof shows that the conjecture implies the validity of   
Theorems \ref{cohomology} and \ref{homotopy} hold without the restriction $m-l\ge 2k^2-k-1$.

\section{Proofs}

The  cohomology ring $H^*(G_{n,k};\mathbb{Z})$ of $G_{n,k}$ is well-known to be generated by 
the Chern classes 
$c_i(\gamma_{n,k})\in H^{2i} (G_{n,k};\mathbb{Z}), 1\le i\le k,$ of the canonical complex $k$-plane bundle $\gamma_{n,k}$.  Indeed, the cohomology ring has a presentation 
\[H^*(G_{n,k};\mathbb{Z})=\mathbb{Z}[c_1,\ldots, c_k]/\langle h_{n-k+1},\ldots, h_n \rangle\]
as the quotient of the polynomial ring modulo the ideal generated by the elements $h_j, n-k+1\le j\le n,$ 
where $|c_i|=2i$; here $h_r$ is defined as the $2r$-th degree term in the expansion of 
$(1+c_1+\cdots+c_k)^{-1}$.  Under the above isomorphism $c_i$ corresponds to the element $c_i(\gamma_{n,k})\in H^{2i}(G_{n,k};\bz), 1\le i\le k.$
We shall denote by $R_k$ the polynomial algebra $\mathbb{Z}[c_1,\ldots,c_k]$ and

The following are well-known facts concerning the cohomology ring:\\
(1) The cohomology group $H^r(G_{n,k};\mathbb{Z})$ is a free abelian group.  It is zero when $r$ is odd.  
This follows from the fact that $G_{n,k}$ admits a cell-structure with cells only in even dimensions. See, for 
example, \cite{griffiths-harris}.
  
(2)  The elements $h_{j}, n-k+1\le j\le n$, form a {\it regular sequence} in the polynomial algebra $R_k:=\mathbb{Z}[c_1,\ldots,c_k]$ for any $n\ge 2k$. That is, $h_{n-k+1}\ne 0$ and $h_{n-k+r}$ is a not a zero 
divisor in $R_k/\langle h_{n-k+1},\cdots,h_{n-k+r-1}\rangle, 2\le r\le k$.  See for example \cite{bott-tu}.   \\

(3)  The element $c_1^d\ne 0$ where $d=\dim_{\mathbb{C}}G_{n,k}=k(n-k)$.  This follows immediately 
from the fact that $G_{n,k}$ has the structure of a K\"ahler manifold with second Betti number $1$. 
In fact it is known $H^{2d}(G_{n,k};\mathbb{Z})$ is generated by the element $c_k^{n-k}$ and that $c_1^d=Nc_k^{n-k}$ where $N=(d!1!2!\cdots(k-1)!)/((n-k)!\cdots (n-1)!)$.  See \cite[\S14]{fulton}. \\

(4) The natural imbedding $i: G_{n,k}\subset G_{n+1,k}$ and $j:G_{n,k}\subset G_{n+1,k+1},$ defined by 
the natural inclusion of $U(n)$ in $U(n+1)$,  
induce surjections $i^*:H^*(G_{n+1,k};\mathbb{Z})\to H^*(G_{n,k},\mathbb{Z})$ and $j^*:H^*(G_{n+1,k+1};\mathbb{Z})\to H^*(G_{n,k};\mathbb{Z})$ where $i^*(c_r(\gamma_{n+1,k}))\mapsto c_r(\gamma_{n,k}), 1\le r\le k$ and 
$j^*(c_r(\gamma_{n+1,k+1}))=c_r(\gamma_{n,k})$ when $r\le k$ and $j^*(c_{k+1}(\gamma_{n+1,k+1}))=0$.
The homomorphism $i^*$ induces isomorphisms in cohomology in dimensions up to $2(n-k)$
and $j^*$ induces isomorphisms in cohomology in dimensions up to $2k$.  

\noindent
{\it Proof of Theorem \ref{cohomology}:}  
One has an inclusion $U(m-l+k)\subset U(m)$ where a matrix $X\in U(m-l+k)$ corresponds to the matrix in block diagonal form with diagonal blocks $X, I_{k-l}$. (Here $I_{k-l}$ denotes the identity matrix.) This induces an imbedding $G_{m-l+k,k}\subset G_{m,l}$. 
Similarly, since $m-l>n-k,$ we have the inclusion $U(n)\subset U(m-l+k)$ which induces  
an imbedding $G_{n,k}\subset G_{m-l+k,k}$.  These inclusions are merely compositions of appropriate inclusions considered in Fact (4) above.  
Let $\alpha:H^*(G_{m,l},\mathbb{Z})\to H^*(G_{m-l+k,k};\mathbb{Z})$ and $\beta:H^*(G_{m-l+k,k};\mathbb{Z})\to H^*(G_{n,k};\mathbb{Z})$ be the inclusion-induced homomorphisms.  It follows from Fact (4) that $\beta(c_i(\gamma_{m-l+k,k}))=c_i(\gamma_{n,k}), i\le k$.
Also,  $\alpha(c_i(\gamma_{m,l}))=c_i(\gamma_{m-l+k,k}),i\leq k$.   
Then we obtain an endomorphism of the graded ring $\alpha\circ\phi\circ \beta$ of $H^*(G_{m-l+k,k})$ where $\phi:H^*(G_{n,k};\mathbb{Z})\to H^*(G_{m,l};\mathbb{Z})$ is any graded ring homomorphism.  Note that 
our hypothesis on $k,l,m,n$ implies that $\dim G_{n,k}<\dim G_{m,l}$. Hence by Fact (3) above, 
$\phi(c_1(\gamma_{n,k}))=0$.  Therefore $\alpha\circ\phi \circ\beta(c_1(\gamma_{m-l+k,k}))=0$. 
Tensoring with $\mathbb{Q}$ we obtain a graded $\mathbb{Q}$-algebra endomorphism $H^*(G_{m-l+k,k};\mathbb{Q})\stackrel{\alpha\circ\phi\circ\beta}{\longrightarrow}H^*(G_{m-l+k,k};\mathbb{Q})$ which vanishes in 
degree $2$.  
Our hypothesis that $m-l\geq 2k^2-k-1$ or $k\le 3$ implies, by \cite{glover-homer}, that this endomorphism is zero in positive dimensions. \hfill $\Box$

We remark that Theorem \ref{cohomology} and the above proof hold when the coefficient ring $\mathbb{Z}$ 
is replaced by any subring of $\mathbb{Q}$ through out.  If $\phi$ is induced by a continuous map $f$, then 
$H^*(f;R)$ is zero for any commutative ring $R$.

Before taking up the proof of Theorem \ref{homotopy}, we recall the relation between the homotopy 
class of a map and the homomorphism it induces in cohomology with rational coefficients. 
We assume familiarity with basic notions in the theory of rational homotopy as in \cite{griffiths-morgan}.  (For a comprehensive treatment see \cite{fht}.) 
 
  Let $X$ be any simply connected finite CW complex and let $X_0$ denote its rationalization.  Thus 
  $\wt{H}^*(X_0;\mathbb{Z})\cong \wt{H}^*(X;\mathbb{Q}).$  If $f:X\to Y$ 
  is a continuous map of such spaces, then there exists a rationalization  of $f$, namely a continuous 
  map $f_0:X_0\to Y_0$ such that $f_0^*:\wt{H}^*(Y_0;\mathbb{Z})\to \wt{H}^*(X_0;\mathbb{Z})$ is the same as 
  $f^*:\wt{H}^*(Y;\mathbb{Q})\to \wt{H}^*(X;\mathbb{Q})$.   Denoting the minimal model of $X$ by 
$\mathcal{M}_X$, one has a bijection $[X_0,Y_0]\cong [\mathcal{M}_Y,\mathcal{M}_X], [h]\mapsto [\Phi_h]$ where on the left we have homotopy classes of continuous maps $X_0\to Y_0$ and on the right we have homotopy classes of differential graded commutative algebra homomorphisms of the minimal models $\mathcal{M}_Y\to \mathcal{M}_X$.  
In the case when $X=U(n)/(U(n_1)\times \cdots\times U(n_r)) $ is a complex flag manifolds, one knows that $X$ is K\"ahler and hence is {\it formal}, that is,  
there exists a morphism of differential graded commutative algebras $\rho_X:\mathcal{M}_X\to H^*(X;\mathbb{Q})$ which induces isomorphism in cohomology, 
where $H^*(X;\mathbb{Q})$ is endowed with the zero differential.  Moreover, it is known that when both 
$X$ and $Y$ are complex flag manifolds, any continuous map $f:X\to Y$ is {\it formal}, that is,  
the homotopy class of the morphism $f_0:X_0\to Y_0$ is determined by  
the graded $\mathbb{Q}$-algebra homomorphisms $h^*:H^*(Y;\mathbb{Q})\to H^*(X;\mathbb{Q})$.  More precisely, we have the following result.

\begin{theorem}(\cite[Theorem 1.1]{glover-homer2})\label{rigid}
Let $X,Y$ be complex flag manifolds.  Then $[h]\mapsto H^*(h;\mathbb{Q})$ establishes an isomorphism from $[X_0,Y_0]$ to the set of graded $\mathbb{Q}$-algebra 
homomorphisms   $H^*(Y;\mathbb{Q})\to H^*(X;\mathbb{Q})$. \hfill $\Box$
\end{theorem}

We now turn to the proof of Theorem \ref{homotopy}.\\

\noindent
{\it Proof of Theorem \ref{homotopy}:} 
By Theorem  \ref{cohomology} we know that any such $f^*$ is the trivial homomorphism 
(which is identity in degree zero and is zero in positive dimensions).   By the above theorem $f_0$ is null-homotopic.   This proves the second 
statement of Theorem \ref{homotopy}.   
The first statement follows from the second since there exists, up to homotopy,  
at most finitely many continuous map $f:G_{m,l}\to G_{n,k}$ having the {\it same} rationalisation $f_0$. (See \cite[\S12]{sullivan}.)
This completes the proof of Theorem \ref{homotopy}.  \hfill $\Box$

Next we turn to the proof of Theorem \ref{invariantsubspace}.  In the following proof, we use cohomology with rational coefficients although one may use integer coefficients.  

 We shall write $M, N$ respectively for $G_{m,l}$ and $G_{n,k}$.  Suppose that $1\le k<l$, $m-l\ge n-k$.  As usual we assume  that $2k\le n, 2l\le m$.  
Let 
$V\subset M\times N$ be the subspace $V:=\{(x,y)\in M\times N\mid 
y\subset x\}\subset M\times N$. (Recall that $\mathbb{C}^n=\mathbb{C}^n\oplus 0\subset \mathbb{C}^m$.)  
One has a map $\pi:V\to N$ that sends $(x,y)\in V$ to $y\in N$.  This is the projection of a fibre 
bundle over $N$ with fibre space $G_{m-k,l-k}$.  To see this, regard 
$V$ as a  
submanifold of the complex flag manifold $F=U(m)/U(k)\times U(l-k)\times U(m-l)=\{(A,B)\mid \dim_\bc A=k, \dim_\bc B=l-k, A\perp B, A,B\subset \bc^m\}$ where a point $(x,y)\in V$ is identified with the point $(y, x')\in F$ 
where $x'$ is the orthogonal complement of $y$ in $x$ so that $x'\perp y$ and $x=x'+y$. 
The projection map $p: F\to G_{m,k}$, defined as $(A,B)\mapsto A\in G_{m,k}$, of the $G_{m-k,l-k}$-bundle $\theta$ 
over $G_{m,k}$ maps $V$ onto $G_{n,k}\subset G_{m,k}$. In fact $V=p^{-1}(G_{n,k})$ and so $\pi:V\to G_{n,k}$ is the projection of the bundle $\theta|_{G_{n,k}}$.

As usual we denote by $[N]$ the generator of the top cohomology group $H^{2k(n-k)}(N;\bq)$.

\noindent 
\begin{lemma} Let $c=${ \em codim}${}_{M\times N}V=2k(m-l).$  
Let $v\in H^c(M\times N;\bq)$ denote the cohomology class dual to $j:V\hookrightarrow M\times N$.  Then $v\cup [N]\ne 0$ in $H^*(M\times N;\bq)$. 
\end{lemma}
\begin{proof}
The cohomology class $[N]$ is dual to the submanifold $i: M\hookrightarrow M\times N$ where $i(x)=(x,\bc^k), x\in M$.  First we shall show that $i(M)$ intersects $V$ transversely.  Note that $i(M)\cap V=\{(x,\bc^k)\mid \bc^k\subset x\subset \bc^m\}\cong G_{m-k,l-k}$, which is the fibre over the point $\bc^k\in N$ of the 
bundle projection $\pi:V\to N$.  Therefore $T_{i(x)}V/T_{i(x)}(V)\cap T_{i(x)} i(M))\cong T_{\bc^l}N$.   Since 
$T_{i(x)}(M\times N)/T_{i(x)}M\cong T_{\bc^l}N$, follows that 
$i(M)$ intersects $V$ transversely. 
Therefore $v\cup [N]$ is dual to the submanifold $V\cap i(M)\subset M\times N$.  Since $V\cap i(M)\cong 
G_{m-k,l-k}\subset G_{m,l}=M$ represents a non-zero homology class in $H_{2(l-k)(m-l)}(M;\bq)\cong H_{2(l-k)(m-l)}(M;\bq)\otimes H_0(N;\bq)\subset H_{2(l-k)(m-l)}(M\times N;\bq)$, its Poincar\'e dual, which equals $v\cup [N]$, is a non-zero cohomology class in $H^{2d}(M\times N;\bq)$ where $d=k(m-l)+k(n-k)$. \end{proof}

\noindent
{\it Proof of Theorem \ref{invariantsubspace}:}  Consider the map $\phi:=id\times f:M\times M\to M\times N$ defined as  $\phi(x,y)=(x,f(y))$.  Denote by $\delta:M\to M\times M$ the diagonal map. 

We need to show that $\phi(\delta(M))\cap V\ne \emptyset$.  

Let $v\in H^*(M\times N;\mathbb{Q})$ denote the cohomology class dual to the manifold $V\subset M\times N$ 
and let $\Delta\in H^*(M\times M;\mathbb{Q})$ denote the diagonal class, i.e., the class dual to $\delta(M)\subset M\times M$.  
As is well-known $v$ is in the image of the inclusion-induced homomorphism $H^*(M\times N, M\times N\setminus V;\mathbb{Q})\to H^*(M\times N;\mathbb{Q})$. (See for example \cite[Chapter 11]{ms}.)   Using the naturality of cup-products and by considering the bilinear map $H^*(M\times N,M\times N\setminus \phi(\delta(M));\mathbb{Q})\otimes H^*(M\times N, M\times N\setminus V;\mathbb{Q})\stackrel{\cup}{\to} H^*(M\times N, M\times N\setminus (V\cap\phi(\delta(M)));\mathbb{Q})$ induced by the inclusion maps, it follows that if $V\cap \phi(\delta(M))=\emptyset$, then  $v\cup w=0$ for any $w\in H^+(M\times N,M\times N\setminus \phi(\delta(M));\bq).$ (See \cite[\S6, Chapter 5]{spanier}.)   In particular, this holds for the class $w$ that maps to the cohomology 
class $\alpha_f$ dual to the submanifold $\phi(\delta(M))\hookrightarrow M\times N$ under the inclusion-induced map 
$H^{2k(n-k)}(M\times N,M\times N\setminus \phi(\delta(M));\bq)\to H^{2k(n-k)}(M\times N;\bq)$.  Thus $v\cup 
\alpha_f=0.$  

On the other hand, $\mu_{M\times N}\cap\alpha_f=\phi_*(\delta_*(\mu_{M}))$.  Our hypothesis 
on $k,l,m,n$ implies, by Theorem \ref{cohomology}, that $\phi_*$ does not depend on $f$.  In particular, taking $f=c$, the constant map sending $M$ to $\bc^k\in N$, we obtain $\phi\circ \delta=i:M\hookrightarrow M\times N$ considered in the previous lemma. So $\phi_*\delta_*(\mu_{M})=i_*(\mu_M)$ and we have $\alpha_f=[N]$.  By 
the above lemma we have $v\cup \alpha_f=v\cup [N]\ne 0$, a contradiction. 
This completes the proof. \hfill $\Box$

 We conclude this paper with the following remark.
Suppose that $\dim (G_{n,k})\le \dim G_{m,l}$ and let $f:G_{m,l}\to G_{n,k}$ 
be a holomorphic map where we assume that $k\le n/2, l\le m/2$.  When $\dim (G_{n,k})=\dim G_{m,l}$ and  $k>1$, so that $G_{n,k}$ is not the projective space, it is was proved by Paranjape and Srinivas \cite{ps} that 
if $f$ is not a constant map, then $(n,k)=(m,l)$ and $f$ is an {\it isomorphism} of varieties. 

Suppose that   $\dim G_{n,k}<\dim G_{m,l}$.  We claim that any holomorphic map $f:G_{m,l} \to G_{n,k}$ is a constant map.  Indeed, the Picard group $Pic(G_{m,l})$ of $G_{m,l}$ of the isomorphism classes of complex (equivalently algebraic or holomorphic) line bundles is isomorphic to $H^2(G_{m,l};\bz)\cong \bz$ via the first  Chern class.  It is generated by the bundle $\xi_{m,l}:=\det(\gamma_{m,l})$.  The dual bundle $\xi_{m,l}^\vee$ is a very ample bundle (or a positive line bundle in the sense of Kodaira).    Note that any holomorphic map between non-singular complex projective manifolds is a 
morphism of algebraic varieties. Now our claim is a consequence of the 
following more general observation.

\begin{lemma} \label{ample}
Let $f:X\to Y$ be a morphism between two complex projective varieties where $Pic(X)$, group of isomorphism 
class of algebraic line bundles over $X$, is isomorphic to the infinite cyclic group. If $\dim X>\dim Y$, then $f$ 
is a constant morphism. 
\end{lemma}
\begin{proof}  Suppose that $f$ is a non-constant morphism.  Then there exists a projective curve $C\subset X$ 
such that $f|_C$ is a finite morphism.
Let $\xi$ be a very ample line bundle over $Y$ and let $\eta=f^*(\xi)$.  Since $\xi$ is very ample, it is 
generated by its (algebraic) sections and so it follows that $\eta$ is also generated by sections.  Since $f|_C$ is a finite morphism, we see that $\eta|_C$ is ample, that is, some positive tensor power of $\eta|_C$ is very ample.  In particular $\eta$ is not trivial.  
Denote by $\omega$ the ample generator of $Pic(X)\cong \bz$ and let $\eta=\omega^{\otimes r}$ for some $r$.  Since $\eta$ is generated by its sections, we have $r\ge 0$.   Since $\eta$ is non-trivial, $r\ne 0$.  It follows that $r>0$ and 
$\eta$ is ample.  

On the other hand, since $\dim X>\dim Y$, some fibre $Z$ of $f$ is positive dimensional and the bundle $\eta|_Z$ 
is trivial.  This is a contradiction since the restriction of an ample bundle to a positive dimensional subvariety is ample and non-trivial.   
\end{proof}
 
\noindent
{\bf Acknowledgments:}  We thank Prof. D. S. Nagaraj for his help in the formulation of Lemma \ref{ample}.   
   
% In particular, it is generated by (holomorphic) sections, and, since the Picard group is infinite cyclic, any line bundle that is generated by sections is either trivial or is isomorphic to a positive tensor power of $\xi_{n,k}^\vee$.  Suppose that $f$ is holomorphic.  
% Then $f^*(\xi_{n,k}^\vee)=:\eta$ is also generated by sections.  If $f$ is not constant, then one can find a projective curve $C$ such that $f|_C$ 
%is a finite morphism.  In particular, $\eta|_C$ is an ample bundle over $C$.  Hence $f^*(\xi_{n,k}^\vee)$ 
%cannot be the trivial bundle and so it has to isomorphic to a positive tensor power of $\xi_{m,l}^\vee$.  It follows that $f^*(c_1(\gamma_{n,k}))\neq 0$.  On the other hand, 
 %$f^*:\bz\cong H^2(G_{n,k};\bz)\to H^2(G_{m,l};\bz)\cong \bz$ is zero since the height of $c_1(\gamma_{n,k})$, which equals the dimension $\dim G_{n,k},$ is {\it less} than that of $c_1(\gamma_{m,l})$.  Therefore 
 %$f$ has to be a constant map.

\end{document}